\documentclass[12pt, a4paper]{amsart}
\usepackage{amsfonts, amsmath, amssymb, amsthm, color}
\usepackage[hmargin=2.75cm, vmargin=2.9cm]{geometry}

\newtheorem{thm}{Theorem}[section]
\newtheorem{lemma}[thm]{Lemma}
\newtheorem{prop}[thm]{Proposition}
\newtheorem{cor}[thm]{Corollary}

\theoremstyle{definition}

\newcommand{\ga}{\gamma}
\newcommand{\ep}{\epsilon}
\newcommand{\vph}{\varphi}
\newcommand{\pa}{\partial}

\newcommand{\N}{\mathbb{N}}
\newcommand{\R}{\mathbb{R}}
\newcommand{\mcd}{\mathcal{D}}

\newcommand{\wtu}{\widetilde{U}}
\newcommand{\wtv}{\widetilde{V}}

\renewcommand{\(}{\left(}
\renewcommand{\)}{\right)}

\begin{document}
\title[non-degeneracy of the Lane-Emden system]{A perturbative approach to \\ non-degeneracy of the Lane-Emden system}

\author[S. Kim]{Seunghyeok Kim}
\address{Department of Mathematics and Research Institute for Natural Sciences, College of Natural Sciences, Hanyang University, 222 Wangsimni-ro Seongdong-gu, Seoul 04763, Republic of Korea}
\email{shkim0401@hanyang.ac.kr shkim0401@gmail.com}

\author[A. Pistoia]{Angela Pistoia}
\address{Dipartimento SBAI, Universit\`{a} di Roma ``La Sapienza", via Antonio Scarpa 16, 00161 Roma, Italy}
\email{pistoia@dmmm.uniroma1.it}

\begin{abstract}
We consider ground state solutions of the critical Lane-Emden system
$$\begin{cases}-\Delta u = v^p &\text{in } \R^n,\\
-\Delta v = u^q &\text{in } \R^n,\\
u,v >0\ &\text{in } \R^n,\end{cases}
$$
where $n \ge 3$ and $p,q>0$  and $(p,q)$ belongs to the {\em critical hyperbola}
$\frac{1}{p+1} + \frac{1}{q+1} = \frac{n-2}{n}.$
We prove that they are {\em non-degenerate} when either $(p,q)$ is close to $(1,{n+4\over n-4})$ (if $n\ge5$) or
$(p,q)$ is close to $({n+2\over n-2},{n+2\over n-2})$ (if $n\ge3$).
\end{abstract}

\date{\today}
\subjclass[2010]{35J47, 35B40}
\keywords{Lane-Emden system, critical hyperbola, non-degenerate solution}
\thanks{S. Kim was partially supported by Basic Science Research Program through the National Research Foundation of Korea(NRF) funded by the Ministry of Education (NRF2017R1C1B5076384).
A. Pistoia was partially supported by Fondi di Ateneo ``Sapienza" Universit\`a di Roma (Italy).}
\maketitle

\allowdisplaybreaks
\numberwithin{equation}{section}
\maketitle

\section{Introduction}
We consider the critical Lane-Emden system
\begin{equation}\label{eq-LEs}
\begin{cases}
-\Delta u = v^p &\text{in } \R^n,\\
-\Delta v = u^q &\text{in } \R^n,\\
u,v >0\ &\text{in } \R^n
\end{cases}
\end{equation}
where $n \ge 3$ and $p,q>0$  and $(p,q)$ belongs to the {\em critical hyperbola}
\begin{equation}\label{eq-hyp}
\frac{1}{p+1} + \frac{1}{q+1} = \frac{n-2}{n}.
\end{equation}

For $s \ge 1$, let $\mcd^{2,s}_0(\R^n)$ be the completion of $C^{\infty}_c(\R^n)$ with respect to the norm $\|\Delta \cdot\|_{L^s(\R^n)}$.
In \cite[Corollary I.2]{Li}, Lions found a positive ground state
\[(U,V) \in \mcd^{2, {p+1 \over p}}_0(\R^n) \times \mcd^{2, {q+1 \over q}}_0(\R^n)\]
of \eqref{eq-LEs}, by transforming it into an equivalent scalar equation
\begin{equation}\label{eq-LEsc}
(-\Delta) \(|\Delta u|^{{1 \over p}-1} (-\Delta u)\) = |u|^{q-1}u \quad \text{in } \R^n
\end{equation}
and employing a concentration-compactness argument to the associated minimization problem
\begin{equation}\label{eq-K_pq}
K_{p,q} = \inf \left\{\|\Delta u\|_{L^{p+1 \over p}(\R^n)}: \|u\|_{L^{q+1}(\R^n)} = 1 \right\}
= \inf_{u \in \mcd^{2, {p+1 \over p}}_0(\R^n) \setminus \{0\}} \frac{\int_{\R^n} |\Delta u|^{p+1 \over p}}{(\int_{\R^n} |u|^{q+1})^{\frac{p+1}{p(q+1)}}}.
\end{equation}
As shown by Alvino et al. \cite{ALT}, it is always radially symmetric and decreasing in $r = |x|$, after a suitable translation.
Moreover, Hulshof and Van der Vorst \cite{HV} proved that a positive radial solution of \eqref{eq-LEs} is unique up to scalings.

\medskip
The present paper deals with non-degeneracy of ground state solutions $(U,V)$ to \eqref{eq-K_pq}
(for which we may assume that $U(0)=1$ without loss of generality).
The invariance of the system under scaling and translations leads to natural solutions of the linearized system around the radial solution $(U,V)$. More precisely, the functions
$$ (U_{\delta,\xi}(x),V_{\delta,\xi}(x)) := \(\delta^{2(p+1) \over pq-1} U(\delta(x-\xi)), \delta^{2(q+1) \over pq-1} V(\delta(x-\xi))\) \quad \hbox{for any}\ \delta>0,\ \xi\in\R^n$$
are solutions to system \eqref{eq-LEs}. If we differentiate the system
$$\begin{cases}
-\Delta U_{\delta,\xi} = V^p_{\delta,\xi} &\text{in } \R^n,\\
-\Delta V_{\delta,\xi} = U^q_{\delta,\xi} &\text{in } \R^n
\end{cases}
$$
with respect to the parameters at $\delta=1$ and $\xi=0,$
we immediately see that the $(n+1)$ linearly independent functions
\[Z_0(x):= \(x\cdot\nabla U + \frac {2(p+1)}{pq-1} U, x\cdot\nabla V + \frac {2(q+1)}{pq-1} V\) \quad \hbox{and} \quad  Z_i(x):= \(\frac{\pa U}{\pa x_i}, \frac{\pa V}{\pa x_i}\)\]
for $i=1,\dots,n$ solves the linear system
\begin{equation}\label{lin}\begin{cases}
-\Delta \phi = pV^{p-1}\psi &\text{in } \R^n,\\
-\Delta \psi = qU^{q-1}\phi &\text{in } \R^n.\\
\end{cases}
\end{equation}
We say that $(U,V)$ is {\em non-degenerate} if all weak solutions to the linear system \eqref{lin} such that $\lim_{|x|\to\infty} (\phi(x),\psi(x)) = (0,0)$ are linear combinations of $Z_0,Z_1,\dots,Z_n.$

\medskip
The non-degeneracy of the solutions of system \eqref{eq-LEs} is a key ingredient in understanding the blow-up phenomena of solutions to the Lane-Emden systems with critical growth.
Therefore,  it is quite natural to ask the following question:
$$\hbox
{\em (Q) \quad  Are ground states $(U,V)$ non-degenerate?}$$

Here, we face the above question, and we give a positive answer in a perturbative setting.
It would be extremely interesting to prove or to disprove non-degeneracy of ground state solutions when $(p,q)$ ranges along all the critical hyperbola \eqref{eq-hyp}.

\medskip
Our main result is:
\begin{thm}\label{main}
There is a small number $\ep > 0$ such that if either $|p-1| \le \ep_0$ (with $n\ge 5$) or
$|p-{n+2\over n-2}| \le \ep$ (with $n\ge3$),
then the unique positive solution $(U,V)$ of \eqref{eq-LEs} (with $U(0) = 1$) is non-degenerate.
\end{thm}
Theorem \ref{main} follows immediately by Corollaries \ref{cor1} and \ref{cor2}, proved in Sections \ref{sec_p1} and \ref{sec_pn}, respectively.
The idea of the proof stems from the simple fact that if $p$ is close to $1$ or close to $n+2\over n-2$,
system \eqref{eq-LEs} is formally close to a single equation whose solutions are non-degenerate.
In particular, if $p$ is close to 1, then $q$ is close to
$ {n+4\over n-4} $ and system \eqref{eq-LEs} can be regarded as a perturbation of the Paneitz-Branson equation
\[(-\Delta)^2u=u^{n+4\over n-4} \quad \text{in } \R^n,\]
while if $p$ is close to ${n+2 \over n-2}$, then $q$ is also close to ${n+2 \over n-2}$ and system \eqref{eq-LEs} becomes a perturbation of the Yamabe equation
\[-\Delta u=u^{n+2\over n-2} \quad \text{in } \R^n.\]
Therefore, to prove our result, we will follow a perturbation argument, which has been successfully applied in various problems like the pseudo-relativistic Hartree equations \cite{Le},
the fractional Schr\"odinger equations \cite{FV} and the Choquard equations \cite{Xi}.
The most challenging part of the proof is to show rigorously that the linearized system \eqref{lin} is close to the corresponding linearized (single) equation,
because sophisticated uniform estimates in $p$ and $q$ of the decay of ground state solutions are required.

\medskip \noindent \textbf{Notations.}
\begin{itemize}
\item[-] For $a \in \R$, let $a_+ = \max\{a,0\}$.
\item[-] For any $x \in \R^n$ and $R > 0$, let $B_R(x) = \{y \in \R^n: |y-x| < R\}$.
\item[-] For a set $D \subset \R^n$, let $\chi_D$ be the characteristic function of $D$.
\item[-] The letters $C$ and $c$ denote positive numbers independent of $p$ that may vary from line to line and inside the same line.
\end{itemize}

\section{Non-degeneracy of the Lane-Emden system near $p = 1$}\label{sec_p1}
The main results of this section are Theorem \ref{thm-deg} and Corollary \ref{cor1}.
To prove them, we will use the following well-known uniqueness and non-degeneracy results about the fourth-order critical equation
\begin{equation}\label{eq-LEbi}
\begin{cases}
(-\Delta)^2 u = u^{n+4 \over n-4} &\text{in } \R^n,\\
u > 0 &\text{in } \R^n
\end{cases}
\end{equation}
for $n \ge 5$.
\begin{prop}\label{prop-lin-bi}
\textnormal{(1) (uniqueness)} Any smooth solution of \eqref{eq-LEbi} is expressed as
\begin{equation}\label{eq-w_lam}
w_{\delta,\xi}(x) := c_n \(\frac{\delta}{\delta^2+|x-\xi|^2}\)^{n-4 \over 2}
\end{equation}
for some $\delta > 0$, $\xi \in \R^n$ and $c_n = [n(n-4)(n-2)(n+2)]^{-{n-4 \over 8}}$.

\medskip \noindent \textnormal{(2) (non-degeneracy)}
The solution space of the linear equation
\begin{equation}\label{eq-lin-bi}
(-\Delta)^2 \phi = \(\frac{n+4}{n-4}\) u^{8 \over n-4} \phi \quad \text{in } \R^n, \quad \phi \in \mcd^{2,2}_0(\R^n)
\end{equation}
is spanned by
\[\frac{\pa u}{\pa x_1},\, \cdots,\, \frac{\pa u}{\pa x_n} \quad \text{and} \quad x \cdot \nabla u + \(\frac{n-4}{2}\)u.\]
\end{prop}
\begin{proof}
Results (1) and (2) have been proved by Lin \cite{Lin} and Lu and Wei \cite{LW}, respectively.
\end{proof}

\subsection{A compactness result}\label{subsec-cpt}
The following is our main result in this subsection.
\begin{prop}\label{prop-cpt}
Suppose that $n \ge 5$. Let $\{p_k\}_{k=1}^{\infty}$ be a sequence such that $p_k \in (\frac{2}{n-2}, \frac{n}{n-2})$ for all $k \in \N$ and $p_k \to 1$ as $k \to \infty$,
Also, let $\{(U_{p_k}, V_{p_k})\}_{k=1}^{\infty}$ be a sequence of the unique positive ground states of \eqref{eq-LEs} with $p = p_k$ such that $U_{p_k}(0) = 1$.
Then we have that
\[(U_{p_k}, V_{p_k}) \to (U_1, V_1) \quad \text{in } \mcd^{2,2}_0(\R^n) \times \mcd^{2,{2n \over n+4}}_0(\R^n) \quad \text{as } k \to \infty.\]
Here $U_1$ is the unique positive solution of \eqref{eq-LEbi} with $U_1(0) = 1$ and $V_1 = -\Delta U_1$ in $\R^n$.
In other words, $(U_1, V_1) = (w_{a_n,0}, -\Delta w_{a_n,0})$ in $\R^n$ where $a_n := c_n^{2 \over n-4}$.
\end{prop}
As we will see, the proofs of the above proposition and Theorem \ref{thm-deg} require uniform upper bound of $(U_{p_k}, V_{p_k})$'s.
It is useful to recall the asymptotic profile of ground state solutions. In \cite{HV}, it has been shown that there exists a pair of positive constants $(\alpha_p, \beta_p)$ such that
\begin{equation}\label{eq-dec}
\lim_{r \to \infty} r^{n-2}\, v(r) = \beta_p\ \hbox{and}\ \begin{cases}
\lim\limits_{r \to \infty} r^{n-2}\, u(r) = \alpha_p &\text{if } p \in (\frac{n}{n-2}, \frac{n+2}{n-2}],\\
\lim\limits_{r \to \infty} \dfrac{r^{n-2}}{\log r}\, u(r) = \alpha_p &\text{if } p = \frac{n}{n-2} ,\\
\lim\limits_{r \to \infty} r^{p(n-2)-2}\, u(r) = \alpha_p &\text{if } p \in (\frac{2}{n-2}, \frac{n}{n-2}).
\end{cases}
\end{equation}

Even though \eqref{eq-dec} depicts the precise asymptotic behavior of $(U_{p_k}, V_{p_k})$ for each $k \in \N$,
it does not readily imply the uniform bound, because the arguments in \cite{HV} do not describe how the sequence $\{(\alpha_{p_k}, \beta_{p_k})\}_{k=1}^{\infty}$ behaves.
In the next two lemmas, we will obtain it by using potential theory. It is not sharp but enough for our purpose.

Without loss of generality, we can assume that $|p_k - 1| \le \ep_0$ for all $k \in \N$ and a small fixed number $\ep_0 > 0$.
Let $q_k$ be the number $q$ determined by \eqref{eq-hyp} with $p = p_k$.

\begin{lemma}\label{lem-cpt-11}
There exist a constant $C > 0$ depending only on $n$ and $\ep_0$ such that
\begin{equation}\label{eq-cpt-11}
U_{p_k}(x) \le 1 \quad \text{and} \quad V_{p_k}(x) \le C \quad \text{for all } x \in \R^n \text{ and } k \in \N
\end{equation}
provided $\ep_0 > 0$ small enough.
\end{lemma}
\begin{proof}
We present the proof by dividing it into 2 steps.

\medskip
\noindent \textsc{Step 1: Uniform boundedness of $K_{p_k, q_k}$.}
Using $w_{a_n,0}$ as a test function of the minimization problem \eqref{eq-K_pq}, we obtain
\begin{equation}\label{eq-K_pq-1}
K_{p_k, q_k} \le \frac{\int_{\R^n} |\Delta w_{a_n,0}|^{p_k+1 \over p_k}}{(\int_{\R^n} w_{a_n,0}^{q_k+1})^{\frac{p_k+1}{p_k(q_k+1)}}}.
\end{equation}
Exploiting the explicit form of $w_{a_n,0}$ and applying the dominated convergence theorem on the right-hand side, we easily deduce that $K_{p_k, q_k}$ is uniformly bounded. In particular,
\begin{equation}\label{eq-K_pq-2}
K_{p_k, q_k} = \(\int_{\R^n} U_k^{q_k+1}\)^{1-\frac{p_k+1}{p_k(q_k+1)}} \le C
\quad \text{or } \quad
\int_{\R^n} U_k^{q_k+1} \le C.
\end{equation}
Note that the second inequality holds, since $q_k \to \frac{n+4}{n-4}$ and so $\frac{(p_k+1)}{p_k(q_k+1)} \to \frac{n-4}{n} < 1$ as $k \to \infty$.

\medskip
\noindent \textsc{Step 2: Uniform boundedness of $(U_{p_k}, V_{p_k})$.}
Because $U_{p_k}(0) = 1$ and $U_{p_k}$ is decreasing in $r$, it holds that $\|U_{p_k}\|_{L^{\infty}(\R^n)} \le 1$ for all $k \in \N$.
By \eqref{eq-LEs}, the Green's representation formula, H\"older's inequality and Step 1, we have
\begin{align*}
V_{p_k}(0) &= \ga_n \int_{B_1(0)} \frac{1}{|y|^{n-2}} U_{p_k}^{q_k}(y) dy
+ \ga_n \int_{\R^n \setminus B_1(0)} \frac{1}{|y|^{n-2}} U_{p_k}^{q_k}(y) dy \\
&\le \ga_n \int_{B_1(0)} \frac{1}{|y|^{n-2}} dy + \ga_n \(\int_{\R^n \setminus B_1(0)} \frac{1}{|y|^{(n-2)(q_k+1)}} dy\)^{1 \over q_k+1}
\(\int_{\R^n \setminus B_1(0)} U_{p_k}^{q_k+1}\)^{q_k \over q_k+1} \\
&\le C
\end{align*}
where $\ga_n := (n(n-2)|B_1(0)|)^{-1}$. The last inequality holds because of \eqref{eq-K_pq-2} and the relation that $(n-2) \frac{2n}{n-4} > n$.
As before, since $V_{p_k}$ is decreasing in $r$, we see that $\|V_{p_k}\|_{L^{\infty}(\R^n)} \le C$ for all $k \in \N$.
\end{proof}

\begin{lemma}\label{lem-cpt-12}
Suppose that an arbitrarily small number $\eta_0 > 0$ is given.
Reducing the size of $\ep_0$ if necessary, one can find a constant $C > 0$ depending only on $n$, $\ep_0$ and $\eta_0$ such that
\[U_{p_k}(x) \le \frac{C}{1+|x|^{n-4-\eta_0}} \quad \text{and} \quad V_{p_k}(x) \le \frac{C}{1+|x|^{n-2}}\]
for all $x \in \R^n$ and $k \in \N$.
\end{lemma}
\begin{proof}
The arguments used here is motivated by the ones in \cite[Section 4]{CK}.
Our proof is relatively simple since we make use of good qualitative properties of the ground states $(U_{p_k}, V_{p_k})$.
We present the proof by dividing it into 2 steps.

\medskip
\noindent \textsc{Step 1: Tail estimate for $U_{p_k}$.}
For any fixed number $R > 0$, we define the functions
\[U_{p_ki} = \chi_{B_R(0)} U_{p_k} \quad \text{and} \quad U_{p_ko} = \chi_{\R^n \setminus B_R(0)} U_{p_k} \quad \text{in } \R^n.\]
We assert that for any given number $\zeta > 0$, there exists a number $R > 0$ depending only on $n$, $\ep_0$ and $\zeta$ such that
\begin{equation}\label{eq-dec-3}
\int_{\R^n \setminus B_R(0)} U_{p_k}^{q_k+1} = \int_{\R^n} U_{p_ko}^{q_k+1} \le \zeta
\end{equation}
for all $k \in \N$.

By Lemma \ref{lem-cpt-11} and elliptic regularity, there exists a pair $(\wtu_1, \wtv_1) \in (C^2(\R^n))^2$ of nonnegative radial functions such that
\begin{equation}\label{eq-conv}
(U_{p_k}, V_{p_k}) \to (\wtu_1, \wtv_1) \quad \text{in } (C^2_{\text{loc}}(\R^n))^2 \quad \text{as } k \to \infty
\end{equation}
along a subsequence. In particular, $(\wtu_1, \wtv_1)$ is a classical solution of \eqref{eq-LEs} with $(p,q) = (1,\frac{n+4}{n-4})$.
Also, since $\wtu_1$ is superharmonic and $\wtu_1(0) = 1$, the maximum principle implies that $\wtu_1 > 0$ in $\R^n$.
In view of Proposition \ref{prop-lin-bi} (1), it holds that $\wtu_1 = w_{a_n,0}$ in $\R^n$ and the convergence in \eqref{eq-conv} is valid for the entire sequence (not just for a subsequence). Summing up,
\begin{equation}\label{eq-conv-2}
(U_{p_k}, V_{p_k}) \to (U_1, V_1) \quad \text{in } (C^2_{\text{loc}}(\R^n))^2 \quad \text{as } k \to \infty
\end{equation}
where we write $(U_1, V_1) = (w_{a_n,0}, -\Delta w_{a_n,0})$ as in the statement of Proposition \ref{prop-cpt}.

Taking the limit $k \to \infty$ on the both sides of \eqref{eq-K_pq-1}, and employing Fatou's lemma, \eqref{eq-K_pq-2} and \eqref{eq-conv-2}, we obtain
\begin{equation}\label{eq-conv-3}
\begin{aligned}
\int_{\R^n} w_{a_n,0}^{\frac{2n}{n-4}} &\le \liminf_{k \to \infty} \int_{\R^n} U_{p_k}^{q_k+1} \le \limsup_{k \to \infty} \int_{\R^n} U_{p_k}^{q_k+1} \\
&= \limsup_{k \to \infty} K_{p_k, q_k}^{\(1-\frac{p_k+1}{p_k(q_k+1)}\)^{-1}}
\le \left[ \frac{\int_{\R^n} |\Delta w_{a_n,0}|^2}{(\int_{\R^n} w_{a_n,0}^{2n \over n-4})^{n-4 \over n}} \right]^{n \over 4} = \int_{\R^n} w_{a_n,0}^{\frac{2n}{n-4}}.
\end{aligned}
\end{equation}
Therefore, all the inequalities must be the equalities. Consequently, applying \eqref{eq-conv-2} and \eqref{eq-conv-3}, we can select $R > 0$ so large that
\[\int_{\R^n} U_{p_ko}^{q_k+1} = \int_{\R^n} U_{p_k}^{q_k+1} - \int_{\R^n} U_{p_ki}^{q_k+1} \to \int_{\R^n} w_{a_n,0}^{\frac{2n}{n-4}}
- \int_{B_R(0)} w_{a_n,0}^{\frac{2n}{n-4}} \le \frac{\zeta}{2} \quad \text{as } k \to \infty.\]
This proves the assertion \eqref{eq-dec-3}.

\medskip
\noindent \textsc{Step 2: Completion of the proof.}
By Green's representation formula, it holds that
\begin{equation}\label{eq-V_po}
\begin{aligned}
\|V_{p_k}\|_{L^{a_1}(\R^n)} &= \ga_n \left\| |\cdot|^{-(n-2)} \ast U_{p_k}^{q_k} \right\|_{L^{a_1}(\R^n)} \\
&\le \ga_n \(\left\| |\cdot|^{-(n-2)} \ast U_{p_ko}^{q_k} \right\|_{L^{a_1}(\R^n)}
+ \left\| |\cdot|^{-(n-2)} \ast U_{p_ki}^{q_k} \right\|_{L^{a_1}(\R^n)}\)
\end{aligned}
\end{equation}
provided that the rightmost side is finite.
Its finiteness is guaranteed for $a_1 > \frac{n}{n-2}$, since \eqref{eq-dec} and Lemma \ref{lem-cpt-11} imply
\begin{equation}\label{eq-dec-2}
(|\cdot|^{-(n-2)} \ast U_{p_ko}^{q_k})(x) \le \frac{C\alpha_{p_k}}{1+|x|^{n-2}}
\quad \text{and} \quad
(|\cdot|^{-(n-2)} \ast U_{p_ki}^{q_k})(x) \le \frac{C}{1+|x|^{n-2}}
\end{equation}
for all $x \in \R^n$ and some constant $C > 0$ depending only on $n$, $\ep_0$ and $R$.\footnote{The inequalities in \eqref{eq-dec-2} are well-known and can be proved as in the proof of \cite[Lemma B.2]{WY}.
We note that the right-hand side of the first inequality in \eqref{eq-dec-2} depends on $k \in \N$, while that of the second one does not.}

On the other hand, the Hardy-Littlewood-Sobolev inequality, H\"older's inequality and \eqref{eq-dec-3} yield
\begin{equation}\label{eq-V_po-2}
\begin{aligned}
\left\| |\cdot|^{-(n-2)} \ast U_{p_ko}^{q_k} \right\|_{L^{a_1}(\R^n)} &\le \|U_{p_ko}^{q_k}\|_{L^{a_2}(\R^n)}
\le \Big\|U_{p_ko}^{p_kq_k-1 \over p_k}\Big\|_{L^{p_k(q_k+1) \over p_kq_k-1}(\R^n)} \Big\|U_{p_ko}^{1 \over p_k}\Big\|_{L^{a_3}(\R^n)} \\
&\le \|U_{p_ko}\|^{p_kq_k-1 \over p_k}_{L^{q_k+1}(\R^n)} \|U_{p_k}\|^{1 \over p_k}_{L^{a_3 \over p_k}(\R^n)}
 \\
&\le C\zeta \||\cdot|^{-(n-2)} \ast V_{p_k}^{p_k}\|^{1 \over p_k}_{L^{a_3 \over p_k}(\R^n)} \\
&\le C\zeta \|V_{p_k}^{p_k}\|^{1 \over p_k}_{L^{a_4}(\R^n)} = C\zeta \|V_{p_k}\|_{L^{a_4p_k}(\R^n)}
\end{aligned}
\end{equation}
where
\[\frac{1}{a_2} = \frac{1}{a_1} + \frac{2}{n}, \quad \frac{1}{a_3} + \frac{p_kq_k-1}{p_k(q_k+1)} = \frac{1}{a_2}, \quad \frac{1}{a_4} = \frac{p_k}{a_3} + \frac{2}{n}\]
and $\zeta > 0$ is an arbitrarily small number.
If we take $\ep_0 > 0$ sufficiently small, then
\[\min\left\{a_2,\, \frac{p_k(q_k+1)}{p_kq_k-1},\, a_3,\, \frac{a_3}{p_k},\, a_4\right\} > 1.\]
Furthermore, we infer from \eqref{eq-hyp} and \eqref{eq-dec} that $a_1 = a_4p_k$ and all the quantities in \eqref{eq-V_po-2} are finite.

Plugging \eqref{eq-V_po-2} into \eqref{eq-V_po} and choosing any $\eta_0' > 0$ small, we find a constant $C > 0$ depending only on $n$, $\ep_0$, $R$ and $\eta_0'$ such that
\[\|V_{p_k}\|_{L^{{n \over n-2} + \eta_0'}(\R^n)} \le C.\]
From the radial symmetry and the decay property of $V_{p_k}$, we deduce
\[V_{p_k}^{{n \over n-2} + \eta_0'}(r)\, r^n \le C\int_{B_r(0)} V_{p_k}^{{n \over n-2} + \eta_0'} \le C\]
where $r = |x| \ge 1$. By combining this with \eqref{eq-cpt-11}, we see that
\[V_{p_k}(x) \le \frac{C}{1+|x|^{n-2-\eta_0'}}\]
for all $x \in \R^n$ and $k \in \N$. As a consequence, we obtain
\[U_{p_k}(x) = \ga_n (|\cdot|^{-(n-2)} \ast V_{p_k}^{p_k})(x) \le C \(|\cdot|^{-(n-2)} \ast \frac{1}{1+|\cdot|^{n-2-\eta_0}}\)(x)
\le \frac{C}{1+|x|^{n-4-\eta_0}},\]
and so
\[V_{p_k}(x) = \ga_n (|\cdot|^{-(n-2)} \ast U_{p_k}^{q_k})(x) \le C \(|\cdot|^{-(n-2)} \ast \frac{1}{1+|\cdot|^{n+4-\eta_0''}}\)(x)
\le \frac{C}{1+|x|^{n-2}}\]
for all $x \in \R^n$ for small $\eta_0,\, \eta_0'' > 0$. This completes the proof.
\end{proof}

\begin{proof}[Completion of the proof of Proposition \ref{prop-cpt}]
By Lemma \ref{lem-cpt-12}, there exists a constant $C > 0$ depending only on $n$ and $\ep_0$ such that
\[\|U_{p_k}\|_{L^{\frac{2n}{n+4} \cdot q_k}(\R^n)} + \|V_{p_k}\|_{L^{2p_k}(\R^n)} \le C.\]
This together with \eqref{eq-LEs} implies uniform boundedness of the sequence $\{(U_{p_k}, V_{p_k})\}_{k=1}^{\infty}$ in the space $\mcd^{2,2}_0(\R^n) \times \mcd^{2,\frac{2n}{n+4}}_0(\R^n)$.

In addition, the uniform decay estimate of $\{(U_{p_k}, V_{p_k})\}_{k=1}^{\infty}$ presented in Lemma \ref{lem-cpt-12} leads
\[\|\Delta U_{p_k}\|_{L^2(\R^n)} = \|V_{p_k}\|_{L^{2p_k}(\R^n)}^{p_k} \to \|V_1\|_{L^2(\R^n)} = \|\Delta U_1\|_{L^2(\R^n)}\]
and
\[\|\Delta V_{p_k}\|_{L^{2n \over n+4}(\R^n)} = \|U_{p_k}\|_{L^{{2n \over n+4} \cdot q_k}(\R^n)}^{q_k} \to \|U_1\|_{L^{2n \over n-4}(\R^n)}^{n+4 \over n-4} = \|\Delta V_1\|_{L^{2n \over n+4}(\R^n)}\]
as $k \to \infty$. As a result, we can invoke \eqref{eq-conv-2} to conclude that
\[(U_{p_k}, V_{p_k}) \to (U_1, V_1) \quad \text{in } \mcd^{2,2}_0(\R^n) \times \mcd^{2,{2n \over n+4}}_0(\R^n) \quad \text{as } k \to \infty,\]
finishing the proof.
\end{proof}

We end this subsection, providing two estimates which will used later.
\begin{lemma}\label{lem-cpt-13}
Suppose that an arbitrarily small number $\eta_2 > 0$ is given.
Reducing the size of $\ep_0$ if necessary, one can find a constant $C > 0$ depending only on $n$, $\ep_0$ and $\eta_2$ such that
\[U_{p_k}(x) \ge \frac{C}{1+|x|^{n-4+\eta_2}} \quad \text{and} \quad V_{p_k}(x) \ge \frac{C}{1+|x|^{n-2}}\]
for all $x \in \R^n$ and $k \in \N$.
\end{lemma}
\begin{proof}
According to \eqref{eq-conv-2}, there exists $C > 0$ depending only on $n$ such that
\[U_{p_k}^{q_k}(x) \ge C \quad \text{for all } x \in B_1(0).\]
Applying the argument in the proof of \cite[Proposition 2]{Vi}, we obtain
\[V_{p_k}(x) \ge \int_{B_1(0)} \frac{\ga_n}{|x-y|^{n-2}} U_{p_k}^{q_k}(y) dy
\ge \frac{C}{1+|x|^{n-2}} \int_{B_1(0)} U_{p_k}^{q_k}(y) dy \ge \frac{C}{1+|x|^{n-2}}\]
and
\[U_{p_k}(x) \ge \int_{B_{\frac{|x|}{2}}(x)} \frac{\ga_n }{|x-y|^{n-2}} V_{p_k}^{p_k}(y) dy \ge \frac{C}{|x|^{n-2}} \int_{B_{\frac{|x|}{2}}(x)} \frac{1}{1+|y|^{n-2+\eta_2}} dy \ge \frac{C}{1+|x|^{n-4+\eta_2}}\]
for $x \in \R^n \setminus B_1(0)$.
\end{proof}

\begin{cor}\label{cor-cpt-12}
Suppose that an arbitrarily small number $\eta_1 > 0$ is given.
Reducing the size of $\ep_0$ if necessary, one can find a constant $C > 0$ depending only on $n$, $\ep_0$ and $\eta_1$ such that
\[|\nabla^l U_{p_k}(x)| \le \frac{C}{1+|x|^{n-4+l-\eta_1}} \quad \text{and} \quad |\nabla^l V_{p_k}(x)| \le \frac{C}{1+|x|^{n-2+l}}\]
for all $x \in \R^n$, $k \in \N$ and $l = 1, 2, 3$.
\end{cor}
\begin{proof}
It immediately follows from Lemmas \ref{lem-cpt-12} and \ref{lem-cpt-13}, and the standard rescaling argument based on elliptic regularity.
\end{proof}

\subsection{Non-degeneracy results near $p = 1$}
Employing the compactness result and pointwise estimates of the sequence $\{(U_{p_k},V_{p_k})\}_{k=1}^{\infty}$ of the unique positive ground states of \eqref{eq-LEs} derived in the previous subsection,
we first deduce a non-degeneracy result of equation \eqref{eq-LEsc} for $p$ near $1$.
\begin{thm}\label{thm-deg}
There exists a small number $\ep_1 \in (0,\ep_0]$ such that if $|p-1| \le \ep_1$ and $U_p$ is the unique positive ground state of \eqref{eq-LEsc} with $U_p(0) = 1$,
then the solution space of the linear equation
\begin{equation}\label{eq-lin}
(-\Delta) \((-\Delta U_p)^{{1 \over p}-1} (-\Delta \phi) \) = pq U_p^{q-1} \phi \quad \text{in } \R^n, \quad \phi \in \mcd^{2,2}_0(\R^n)
\end{equation}
is spanned by
\begin{equation}\label{eq-lin-sol}
\frac{\pa U_p}{\pa x_1},\, \cdots,\, \frac{\pa U_p}{\pa x_n} \quad \text{and} \quad x \cdot \nabla U_p + \(n-2 - \frac{n}{p+1}\) U_p.
\end{equation}
\end{thm}
\begin{proof}
Thanks to Lemma \ref{lem-cpt-12}, Corollary \ref{cor-cpt-12} and elliptic regularity,
the functions listed in \eqref{eq-lin-sol} belong to the space $\mcd^{2,2}_0(\R^n) \cap C^{\infty}(\R^n)$.

For $j = 1, \cdots, n$, each $\frac{\pa U_p}{\pa x_j}$ clearly solve \eqref{eq-lin}. Also, if we set
\[U_{p,\delta}(x) = \delta^{2(p+1) \over pq-1} U_p(\delta x) \quad \text{in } \R^n\]
for each $\delta > 0$, every $U_{p,\delta}$ is a solution of \eqref{eq-LEsc}. Therefore
\begin{equation}\label{eq-U_pmu-2}
\left. \frac{\pa U_{p,\delta}}{d\delta} \right|_{\delta=1} = x \cdot \nabla U_p + \(n-2 - \frac{n}{p+1}\) U_p,
\end{equation}
whose equality holds due to \eqref{eq-hyp}, solves \eqref{eq-lin}.

\medskip
Suppose that there exist
\begin{itemize}
\item[-] a sequence $\{p_k\}_{k=1}^{\infty} \subset [1-\ep_0, 1+\ep_0]$ of numbers tending to $1$ as $k \to \infty$;
\item[-] a sequence $\{U_{p_k}\}_{k=1}^{\infty}$ of the unique positive ground states of \eqref{eq-LEsc} with $p = p_k$ such that $U_{p_k}(0) = 1$;
\item[-] a sequence $\{\phi_k\}_{k=1}^{\infty} \subset \mcd^{2,2}_0(\R^n)$ of solutions of \eqref{eq-lin} with $p = p_k$ and $u = U_{p_k}$
    which cannot be written as a linear combination of the functions
\[\qquad Z_{1p_k} = \frac{\pa U_{p_k}}{\pa x_1},\, \cdots,\, Z_{np_k} = \frac{\pa U_{p_k}}{\pa x_1} \quad \text{and} \quad Z_{0p_k} = x \cdot \nabla U_{p_k} + \(n-2 - \frac{n}{p_k+1}\) U_{p_k}.\]
\end{itemize}
We may assume further that $\|\Delta \phi_k\|_{L^2(\R^n)} = 1$ and
\begin{equation}\label{eq-lin-0}
\int_{\R^n} \Delta \phi_k\, \Delta Z_{0p_k} = \int_{\R^n} \Delta \phi_k\, \Delta Z_{1p_k} = \cdots = \int_{\R^n} \Delta \phi_k\, \Delta Z_{np_k} = 0.
\end{equation}

The rest of the proof is split into 4 steps.

\medskip
\noindent \textsc{Step 1: Uniform boundedness of $\phi_k$'s and $\Delta \phi_k$'s.}
We claim that there exists a constant $C > 0$ depending only on $n$ and $\ep_0$ such that
\begin{equation}\label{eq-phi}
\|\phi_k\|_{L^{\infty}(\R^n)} + \|\Delta \phi_k\|_{L^{\infty}(\R^n)} \le C
\end{equation}
for all $k \in \N$.

Define
\begin{equation}\label{eq-psi_k}
\psi_k = \frac{1}{p_k} (-\Delta U_{p_k})^{{1 \over p_k}-1} (-\Delta \phi_k) = - \frac{1}{p_k} V_{p_k}^{1-p_k} \Delta \phi_k \quad \text{in } \R^n
\end{equation}
for each $k \in \N$. Then \eqref{eq-lin} is rewritten as the linearized equation of system \eqref{eq-LEs}
\begin{equation}\label{eq-lins}
\begin{cases}
-\Delta \phi_k = p_k V_{p_k}^{p_k-1} \psi_k &\text{in } \R^n,\\
-\Delta \psi_k = q_k U_{p_k}^{q_k-1} \phi_k &\text{in } \R^n.
\end{cases}
\end{equation}
Fix any $x_0 \in \R^n$ such that $|x_0| \ge 2$ and set $R = |x_0|$. For any $0 < r' < r \le 1$ and $l \in \N \cap \{0\}$ such that $n > 4l$, it holds that
\begin{equation}\label{eq-CZ-1}
\|\psi_k\|_{W^{2,{2n \over n-4l}}(B_{r'}(x_0))} \le C\(\|\psi_k\|_{L^{2n \over n-4l}(B_r(x_0))} + \|U_{p_k}^{q_k-1} \phi_k\|_{L^{2n \over n-4l}(B_r(x_0))} \)
\end{equation}
and
\begin{equation}\label{eq-CZ-2}
\|\phi_k\|_{W^{2,{2n \over n-4l}}(B_{r'}(x_0))} \le C\(\|\phi_k\|_{L^{2n \over n-4l}(B_r(x_0))} + \|V_{p_k}^{p_k-1} \psi_k\|_{L^{2n \over n-4l}(B_r(x_0))}\)
\end{equation}
provided that the right-hand sides are finite.

By \eqref{eq-psi_k}, Lemmas \ref{lem-cpt-12} and \ref{lem-cpt-13}, $\|\Delta \phi_k\|_{L^2(\R^n)} = 1$ and the Sobolev inequality, we find
\begin{align*}
\|\psi_k\|_{L^2(B_1(x_0))} &\le C \|V_{p_k}^{1-p_k} \Delta \phi_k\|_{L^2(B_1(x_0))} \\
&\le CR^{(n-2)(p_k-1)} \|\Delta \phi_k\|_{L^2(B_1(x_0))} \le CR^{(n-2)(p_k-1)}
\end{align*}
and
\[\|U_{p_k}^{q_k-1} \phi_k\|_{L^2(B_1(x_0))} \le CR^{-(n-4-\eta_0)(q_k-1)} \|\phi_k\|_{L^{2n \over n-4}(B_1(x_0))} \le C R^{-7} \|\Delta \phi_k\|_{L^2(\R^n)} = CR^{-7}.\]
Here $C > 0$ is a constant depending only on $n$, $\ep_0$ and $\eta_0$, and particularly, independent of $x_0$ and $R$.
Hence \eqref{eq-CZ-1} with $l = 0$ shows that
\begin{equation}\label{eq-CZ-3}
\|\psi_k\|_{L^{2n \over n-4}(B_{1/2}(x_0))} \le C \|\psi_k\|_{W^{2,2}(B_{1/2}(x_0))} \le CR^{(n-2)(p_k-1)}.
\end{equation}
On the other hand, it follows from \eqref{eq-CZ-3} that
\[\|V_{p_k}^{p_k-1} \psi_k\|_{L^{2n \over n-4}(B_{1/2}(x_0))} \le R^{(n-2)(1-p_k)} \|\psi_k\|_{L^{2n \over n-4}(B_{1/2}(x_0))} \le C.\]
Thus \eqref{eq-CZ-2} with $l = 1$ gives
\begin{equation}\label{eq-CZ-4}
\|\phi_k\|_{W^{2,{2n \over n-4}}(B_{1/3}(x_0))} \le C\(\|\Delta \phi_k\|_{L^2(\R^n)} + \|V_{p_k}^{p_k-1} \psi_k\|_{L^{2n \over n-4}(B_{1/2}(x_0))} \) \le C.
\end{equation}
Putting \eqref{eq-CZ-3} and \eqref{eq-CZ-4} into \eqref{eq-CZ-1} with $l = 1$, we obtain
\begin{equation}\label{eq-CZ-5}
\|\psi_k\|_{W^{2,{2n \over n-4}}(B_{1/4}(x_0))} \le CR^{(n-2)(p_k-1)}.
\end{equation}
If $5 \le n \le 7$, we have that $W^{2,{2n \over n-4}}(B_r(x_0)) \hookrightarrow L^{\infty}(B_r(x_0))$ for $r > 0$.
Therefore, by means of \eqref{eq-psi_k}, \eqref{eq-CZ-4} and \eqref{eq-CZ-5}, we deduce
\begin{equation}\label{eq-CZ-6}
\|\phi_k\|_{L^{\infty}(\{|x| \ge 2\})} + \|\Delta \phi_k\|_{L^{\infty}(\{|x| \ge 2\})} \le C.
\end{equation}
For higher dimensional case, we repeat the above process to improve integrability of $\psi_k$'s and $\phi_k$'s.
After a finite number of iterations, we obtain \eqref{eq-CZ-6}.
The uniform boundedness of $\phi_k$'s and their Laplacians on the set $\{|x| \le 2\}$ is easier to deduce.
Our claim \eqref{eq-phi} is justified.

\medskip
\noindent \textsc{Step 2: Rough decay estimates of $\phi_k$'s and $\Delta \phi_k$'s.}
We assert that there exists a constant $C > 0$ depending only on $n$ and $\ep_0$ such that
\begin{equation}\label{eq-dec-phi-1}
|\phi_k(x)| \le \frac{C}{1+|x|^{n-4 \over 2}} \quad \text{and} \quad |\Delta \phi_k(x)| \le \frac{C}{1+|x|^{n \over 2}}
\end{equation}
for all $x \in \R^n$ and $k \in \N$. The arguments in this and the next steps are inspired by the proof of \cite[Lemma 3.3]{DKP}.

Fix any $x_0 \in \R^n$ such that $|x_0| \ge 2$ and $R = |x_0|$. Define also
\[\phi_{kR}(x) = R^{n-4 \over 2} \phi_k(Rx) \quad \text{and} \quad \psi_{kR}(x) = R^{n \over 2}\psi_k(Rx) \quad \text{in } \R^n\]
for each $k \in \N$. They solve
\[\begin{cases}
-\Delta \phi_{kR} = p_k (V_{p_k}(Rx))^{p_k-1} \psi_{kR} &\text{in } \R^n,\\
-\Delta \psi_{kR} = q_k R^4 (U_{p_k}(Rx))^{q_k-1} \phi_{kR} &\text{in } \R^n.
\end{cases}\]
For each $t > 1$, set $A_t = \{x \in \R^n: 1/t < |x| < t\}$.
For any $r > r' > 1$ and $l \in \N \cap \{0\}$ such that $n > 4l$, it holds that
\[\|\psi_{kR}\|_{W^{2,{2n \over n-4l}}(A_{r'})} \le C\(\|\psi_{kR}\|_{L^{2n \over n-4l}(A_r)} + R^4 \|(U_{p_k}(R\cdot))^{q_k-1} \phi_{kR}\|_{L^{2n \over n-4l}(A_r)} \)\]
and
\[\|\phi_{kR}\|_{W^{2,{2n \over n-4l}}(A_{r'})} \le C\(\|\phi_{kR}\|_{L^{2n \over n-4l}(A_r)} + \|(V_{p_k}(R\cdot))^{p_k-1} \psi_{kR}\|_{L^{2n \over n-4l}(A_r)}\)\]
provided that the right-hand sides are finite. Besides, we have that $\|\Delta \phi_{kR}\|_{L^2(\R^n)} = 1$.
Hence, arguing as in Step 1, we obtain
\[\|\phi_{kR}\|_{L^{\infty}(\{|x| \ge 2\})} + \|\Delta \phi_{kR}\|_{L^{\infty}(\{|x| \ge 2\})} \le C.\]
Combining this with \eqref{eq-phi}, we conclude that \eqref{eq-dec-phi-1} is true.

\medskip
\noindent \textsc{Step 3: Almost sharp decay estimates of $\phi_k$'s and $\Delta \phi_k$'s.}
Let $\eta_3 > 0$ be any small number.
We will show that there exists a constant $C > 0$ depending only on $n$, $\ep_0$ and $\eta_3$ such that
\begin{equation}\label{eq-dec-phi-2}
|\phi_k(x)| \le \frac{C}{1+|x|^{n-4-\eta_3}} \quad \text{and} \quad |\Delta \phi_k(x)| \le \frac{C}{1+|x|^{n-2-\eta_3}}
\end{equation}
for all $x \in \R^n$ and $k \in \N$.

Fix $k \in \N$, and let $\mu = \frac{n-4}{2}$ and $\nu = \frac{n}{2}$.
For an arbitrary number $\eta > 0$, we define
\[F_{k, \mu, \eta}(x) = \phi_k(x) - \frac{M_{\mu, \eta}}{|x|^{\mu+\eta}}
\quad \text{and} \quad
G_{k, \nu, \eta}(x) = \psi_k(x) - \frac{m_{\nu, \eta}}{|x|^{\nu+\eta}}
\quad \text{in } \{|x| \ge 1\},\]
where $M_{\mu, \eta}$ and $m_{\nu, \eta}$ are large positive numbers determined by their subscripts.
If $R > 1$ is given, we get from \eqref{eq-dec-phi-1} that
\[-\Delta G_{k, \nu, \eta}(x) = q_k U_{p_k}^{q_k-1} \phi_k - \frac{m_{\nu, \eta} (\nu+\eta)(n-2-(\nu+\eta))}{|x|^{\nu+\eta+2}} \le 0 \quad \text{in } \{1 < |x| < R\}\]
and
\[G_{k, \nu, \eta}(x) \le 0 \quad \text{on } \{|x| = 1\}\]
provided that $\nu+\eta < \min\{\mu+5, n-2\}$ and $m_{\nu, \eta}$ is chosen to be large enough.
Hence the maximum principle yields that
\[G_{k, \nu, \eta}(x) \le \max_{\{|x| = R\}} (G_{k, \nu, \eta})_+ \quad \text{in } \{1 < |x| < R\}.\]
Taking $R \to \infty$ and applying \eqref{eq-dec-phi-1} again, we deduce that $G_{k, \nu, \eta}(x) \le 0$, or equivalently,
\[\psi_k(x) \le \frac{m_{\nu, \eta}}{|x|^{\nu+\eta}} \quad \text{in } \{|x| \ge 1\}.\]
Similarly, one can show that $-\psi_k$ has the same upper bound in $\{|x| \ge 1\}$.
Therefore, we improve the decay rate of $\psi_k$ as follows:
\[|\psi_k(x)| \le \frac{m_{\nu, \eta}}{|x|^{\nu+\eta}} \quad \text{in } \{|x| \ge 1\}.\]
Resetting $\nu$ as $\nu+\eta$, we repeat the above procedure with the function $F_{k, \mu, \eta}$ to improve the decay rate of $\phi_k$'s.
This information can be used in further improvement of the decay rate of $\psi_k$'s.
We iterate such a process until we reach \eqref{eq-dec-3}.

\medskip
\noindent \textsc{Step 4: Completion of the proof.}
Eq. \eqref{eq-lin} is read as
\begin{equation}\label{eq-lin-1}
\int_{\R^n} (-\Delta U_{p_k})^{{1 \over p_k}-1} \Delta \phi_k\, \Delta \vph = p_k q_k \int_{\R^n} U_{p_k}^{q_k-1} \phi_k \vph \quad \text{for any } \vph \in C^{\infty}_c(\R^n).
\end{equation}
Also, there exists a function $\phi_{\infty} \in \mcd^{2,2}_0(\R^n)$ such that
\[\phi_k \rightharpoonup \phi_{\infty} \quad \text{in }\mcd^{2,2}_0(\R^n)
\quad \text{and} \quad
\phi_k \to \phi_{\infty},\, \Delta \phi_k \to \Delta \phi_{\infty} \quad \text{a.e. in } \R^n\]
as $k \to \infty$, passing to a subsequence.

Invoking Lemmas \ref{lem-cpt-12} and \ref{lem-cpt-13}, Proposition \ref{prop-cpt} and the dominated convergence theorem, we infer
\begin{equation}\label{eq-lin-2}
\int_{\R^n} (-\Delta U_{p_k})^{{1 \over p_k}-1} \Delta \phi_k\, \Delta \vph
= \int_{\R^n} V_{p_k}^{1-p_k} \Delta \phi_k\, \Delta \vph
\to \int_{\R^n} \Delta \phi_{\infty}\, \Delta \vph
\end{equation}
and
\begin{equation}\label{eq-lin-3}
p_k q_k \int_{\R^n} U_{p_k}^{q_k-1} \phi_k \vph \to \(\frac{n+4}{n-4}\) \int_{\R^n} U_1^{8 \over n-4} \phi_{\infty} \vph
\end{equation}
as $k \to \infty$.

Putting \eqref{eq-lin-1}, \eqref{eq-lin-2} and \eqref{eq-lin-3} together, we conclude that $\phi_{\infty}$ is a solution of \eqref{eq-lin-bi}.
By Proposition \ref{prop-lin-bi} (2), it follows that
\[\phi_{\infty} = \sum_{j=1}^n c_j \frac{\pa U_1}{\pa x_j} + c_0 \left[ x \cdot \nabla U_1 + \(\frac{n-4}{2}\)U_1 \right] \quad \text{in } \R^n.\]

We now assert that $\phi_{\infty} \ne 0$.
By \eqref{eq-dec-phi-2}, we can take $\vph = \phi_k$ in \eqref{eq-lin-1}. Applying the mean value theorem, we observe
\[\int_{\R^n} (-\Delta U_{p_k})^{{1 \over p_k}-1} (\Delta \phi_k)^2 =
\|\Delta \phi_k\|_{L^2(\R^n)}^2 + \int_{\R^n} \(V_{p_k}^{1-p_k}-1\) (\Delta \phi_k)^2 = 1 + o(1)\]
where $o(1) \to 0$ as $k \to \infty$. Since
\[p_k q_k \int_{\R^n} U_{p_k}^{q_k-1} \phi_k^2 \to \(\frac{n+4}{n-4}\) \int_{\R^n} U_1^{8 \over n-4} \phi_{\infty}^2 \quad \text{as } k \to \infty,\]
it holds that
\[\int_{\R^n} U_1^{8 \over n-4} \phi_{\infty}^2 = \frac{n-4}{n+4} \ne 0.\]
Consequently, we have that $\phi_{\infty} \ne 0$ and so $\sum_{j=0}^n |c_j| \ne 0$.
However, \eqref{eq-lin-0} and Corollary \ref{cor-cpt-12} imply
\[\int_{\R^n} \Delta \phi_{\infty}\, \Delta \(\frac{\pa U_1}{\pa x_j}\)
= \int_{\R^n} \Delta \phi_{\infty}\, \Delta \left[ x \cdot \nabla U_1 + \(\frac{n-4}{2}\)U_1 \right] = 0\]
for $j = 1, \cdots, n$. Hence $c_0 = \cdots = c_n = 0$, a contradiction.
This completes the proof of the theorem.
\end{proof}

Theorem \ref{thm-deg} and the equivalence between system \eqref{eq-LEs} and equation \eqref{eq-LEsc} yield the following non-degeneracy result for \eqref{eq-LEs} with $p$ near $1$.
\begin{cor}\label{cor1}
There exists a small number $\ep_1 \in (0,\ep_0]$ such that if $|p-1| \le \ep_1$ and $(U_p,V_p)$ is the unique positive ground state of \eqref{eq-LEs} with $U_p(0) = 1$,
then the solution space of the linear equation
\begin{equation}\label{eq-lins-2}
\begin{cases}
-\Delta \phi = p V_p^{p-1} \psi &\text{in } \R^n,\\
-\Delta \psi = q U_p^{q-1} \phi &\text{in } \R^n,
\end{cases}
\quad \lim_{|x|\to\infty} (\phi(x),\psi(x)) = (0,0)
\end{equation}
is spanned by
\[\(\frac{\pa U_p}{\pa x_1}, \frac{\pa V_p}{\pa x_1}\),\, \cdots,\, \(\frac{\pa U_p}{\pa x_n}, \frac{\pa V_p}{\pa x_n}\)\]
and
\[\(x \cdot \nabla U_p + \(n-2 - \frac{n}{p+1}\) U_p, x \cdot \nabla V_p + \(n-2 - \frac{n}{q+1}\) V_p\).\]
\end{cor}
\begin{proof}
The relationship between \eqref{eq-lin} and \eqref{eq-lins-2} was already explored in \eqref{eq-psi_k} and \eqref{eq-lins}.
Moreover, with the help of the decay assumption of $(\phi,\psi)$ and elliptic regularity, one can argue as in Step 3 of the proof of Theorem \ref{thm-deg} to verify that $\phi \in \mcd^{2,2}_0(\R^n)$.
Therefore, the necessary condition to apply Theorem \ref{thm-deg} is fulfilled and so
\[\phi = \frac{\pa U_p}{\pa x_1},\, \cdots,\, \frac{\pa U_p}{\pa x_n} \quad \text{and} \quad x \cdot \nabla U_p + \(n-2 - \frac{n}{p+1}\) U_p.\]
Suppose that $\phi = \frac{\pa U_p}{\pa x_j}$ for some $j = 1, \cdots, n$. Differentiating the first equation in \eqref{eq-LEs} with respect to $x_j$ and using \eqref{eq-lins-2}, we find
\[\psi = - \frac{1}{p} V_p^{1-p} \Delta \(\frac{\pa U_p}{\pa x_j}\) = \frac{1}{p} V_p^{1-p} \cdot p V_p^{p-1} \frac{\pa V_p}{\pa x_j} = \frac{\pa V_p}{\pa x_j}.\]
Set
\[V_{p,\delta}(x) = \delta^{2(q+1) \over pq-1} V_p(\delta x) \quad \text{in } \R^n.\]
If $\phi = x \cdot \nabla U_p + (n-2 - \frac{n}{p+1}) U_p$, then \eqref{eq-U_pmu-2} shows
\begin{align*}
\psi &= - \frac{1}{p} V_p^{1-p} \left. \Delta \(\frac{\pa U_{p,\delta}}{\pa \delta}\) \right|_{\delta=1}
= \frac{1}{p} V_p^{1-p} \cdot \left. p V_{p,\delta}^{p-1} \frac{\pa V_{p,\delta}}{\pa \delta} \right|_{\delta=1} \\
&= x \cdot \nabla V_p + \(n-2 - \frac{n}{q+1}\) V_p.
\end{align*}
The proof is done.
\end{proof}

\section{Non-degeneracy of the Lane-Emden system near $p = \frac{n+2}{n-2}$}\label{sec_pn}
The main results of this section are Theorem \ref{thm-deg-so} and Corollary \ref{cor2},
whose proof depends on arguments similar to those used in the previous section.
In this time, we use the following well-known uniqueness and non-degeneracy results about the second-order critical equation
\begin{equation}\label{eq-LE-so}
\begin{cases}
-\Delta u = u^{n+2 \over n-2} &\text{in } \R^n,\\
u > 0 &\text{in } \R^n
\end{cases}
\end{equation}
for $n \ge 3$.
\begin{prop}
\textnormal{(1) (uniqueness)} Any smooth solution of \eqref{eq-LE-so} is expressed as
\[w_{\delta,\xi}^*(x) := c_n^* \(\frac{\delta}{\delta^2+|x-\xi|^2}\)^{n-2 \over 2}\]
for some $\delta > 0$, $\xi \in \R^n$ and $c_n^* = [n(n-2)]^{{n-2 \over 4}}$.

\medskip \noindent \textnormal{(2) (non-degeneracy)}
The solution space of the linear equation
\begin{equation}\label{eq-lin-so}
-\Delta \phi = \(\frac{n+2}{n-2}\) u^{4 \over n-2} \phi \quad \text{in } \R^n, \quad \phi \in \mcd^{1,2}_0(\R^n)
\end{equation}
is spanned by
\[\frac{\pa u}{\pa x_1},\, \cdots,\, \frac{\pa u}{\pa x_n} \quad \text{and} \quad x \cdot \nabla u + \(\frac{n-2}{2}\)u.\]
Here, $\mcd^{1,2}_0(\R^n)$ be the completion of $C^{\infty}_c(\R^n)$ with respect to the norm $\|\nabla \cdot\|_{L^2(\R^n)}$.
\end{prop}
\begin{proof}
Result (1) has been proved by Aubin \cite{A}, Talenti \cite{T} and Caffarelli et al. \cite{CGS}.
A proof of (2) can be found in Rey  \cite{R}.
\end{proof}
We will assume that $p$ is {\em slightly smaller} than $\frac{n+2}{n-2}$.
By interchanging the role of $U$ and $V$ and of $p$ and $q$, we can also cover the case that $p$ is {\em slightly bigger} than $\frac{n+2}{n-2}$.
Adapting the arguments in Subsection \ref{subsec-cpt}, we obtain the next two results.
\begin{prop}\label{prop-cpt-so}
Suppose that $n \ge 3$ and $2^* = \frac{n+2}{n-2}$.
Let $\{p_k\}_{k=1}^{\infty}$ be a sequence such that $p_k \in [\frac{n}{n-2}, 2^*]$ for all $k \in \N$ and $p_k \to 2^*$ as $k \to \infty$,
Also, let $\{(U_{p_k}, V_{p_k})\}_{k=1}^{\infty}$ be the sequence of the unique positive ground states of \eqref{eq-LEs} with $p = p_k$ such that $U_{p_k}(0) = 1$.
Then we have that
\[(U_{p_k}, V_{p_k}) \to (U_{2^*}, V_{2^*}) \quad \text{in } \mcd^{1,2}_0(\R^n) \times \mcd^{1,2}_0(\R^n) \quad \text{as } k \to \infty.\]
Here $U_{2^*}$ is the unique positive radial solution of \eqref{eq-LE-so} with $U_{2^*}(0) = 1$ and $V_{2^*} = U_{2^*}$ in $\R^n$.
In other words, $(U_{2^*}, V_{2^*}) = (w_{b_n,0}^*, w_{b_n,0}^*)$ in $\R^n$ where $b_n := c_n^{2 \over n-2}$.
\end{prop}
\begin{proof}
The fact that $V_{2^*} = U_{2^*}$ comes from \cite[Lemma 2.7 and Remark 2.1 (a)]{Sou}.
\end{proof}
\begin{lemma}
Given $\ep_2 > 0$ small enough, we assume that $|p_k - \frac{n+2}{n-2}| \le \ep_2$ for all $k \in \N$.
Then one can find a constant $C > 0$ depending only on $n$, $\ep_2$ and $\eta_1$ such that
\[|\nabla^l U_{p_k}(x)| + |\nabla^l V_{p_k}(x)| \le \frac{C}{1+|x|^{n-2+l}}\]
for all $x \in \R^n$, $k \in \N$ and $l = 1, 2, 3$.
\end{lemma}

By employing the above results and slightly modifying the proof of Theorem \ref{thm-deg}, one can deduce the following theorem.
\begin{thm}\label{thm-deg-so}
There exists a small number $\ep_3 \in (0,\ep_2]$ such that if $|p-\frac{n+2}{n-2}| \le \ep_3$ and $U_p$ is the unique positive ground state of \eqref{eq-LE-so} with $U_p(0) = 1$,
then the solution space of the linear equation
\[(-\Delta) \((-\Delta U_p)^{{1 \over p}-1} (-\Delta \phi) \) = pq U_p^{q-1} \phi \quad \text{in } \R^n, \quad \phi \in \mcd^{2,2}_0(\R^n)\]
is spanned by
\[\frac{\pa U_p}{\pa x_1},\, \cdots,\, \frac{\pa U_p}{\pa x_n} \quad \text{and} \quad x \cdot \nabla U_p + \(n-2 - \frac{n}{p+1}\) U_p.\]
\end{thm}
\begin{proof}
The proof goes along the same way to the proof of Theorem \ref{thm-deg} except Step 4.

In order to facilitate Step 4, we have to prove that $\phi_{\infty}$ is a solution of \eqref{eq-lin-so}.
Taking $k \to \infty$ in \eqref{eq-lins} and using $U_{2^*} = V_{2^*} = w_{b_n,0}^*$ (which was confirmed in Proposition \ref{prop-cpt-so}), we obtain
\[\int_{\R^n} \nabla \phi_{\infty} \cdot \nabla \vph = \(\frac{n+2}{n-2}\) \int_{\R^n} U_{2^*}^{4 \over n-2} \psi_{\infty} \vph\]
and
\[\int_{\R^n} \nabla \psi_{\infty} \cdot \nabla \vph = \(\frac{n+2}{n-2}\) \int_{\R^n} U_{2^*}^{4 \over n-2} \phi_{\infty} \vph.\]
We subtract the second equation from the first equation, and then put $\vph = \phi_{\infty} - \psi_{\infty}$. Then we get
\[0 \le \int_{\R^n} |\nabla (\phi_{\infty} - \psi_{\infty})|^2 = - \(\frac{n+2}{n-2}\) \int_{\R^n} U_{2^*}^{4 \over n-2} (\phi_{\infty} - \psi_{\infty})^2 \le 0.\]
Therefore, $\phi_{\infty} = \psi_{\infty}$ solves \eqref{eq-lin-so}. The rest of the proof remains the same.
\end{proof}

Arguing as in the proof of Corollary \ref{cor1}, we derive the following result from the previous theorem.
\begin{cor}\label{cor2}
There exists a small number $\ep_3 \in (0,\ep_2]$ such that if $|p-{n+2\over n-2}| \le \ep_3$  and $(U_p,V_p)$ is the unique positive ground state of \eqref{eq-LEs} with $U_p(0) = 1$,
then the solution space of the linear equation
\[\begin{cases}
-\Delta \phi = p V_p^{p-1} \psi &\text{in } \R^n,\\
-\Delta \psi = q U_p^{q-1} \phi &\text{in } \R^n,
\end{cases}
\quad \lim_{|x|\to\infty} (\phi(x),\psi(x)) = (0,0)\]
is spanned by
\[\(\frac{\pa U_p}{\pa x_1}, \frac{\pa V_p}{\pa x_1}\),\, \cdots,\, \(\frac{\pa U_p}{\pa x_n}, \frac{\pa V_p}{\pa x_n}\)\]
and
\[\(x \cdot \nabla U_p + \(n-2 - \frac{n}{p+1}\) U_p, x \cdot \nabla V_p + \(n-2 - \frac{n}{q+1}\) V_p\).\]
\end{cor}

\end{document}